\documentclass[letterpaper, 10 pt, conference]{ieeeconf}

\IEEEoverridecommandlockouts
\usepackage{amsmath}
\usepackage{amssymb}

\usepackage{amsthm}
\usepackage{graphicx}
\usepackage{mathtools}

\newtheorem{prop}{Proposition}

\usepackage[hidelinks]{hyperref}
\usepackage{colortbl}
\usepackage{xcolor}
\usepackage[linesnumbered,ruled,vlined]{algorithm2e}

\SetCommentSty{mycommfont}
\SetKwInput{KwInput}{Input}
\SetKwInput{KwOutput}{Output}

\begin{document}
\title{ \LARGE \bf Stagewise Newton Method for Dynamic Game Control \\ with Imperfect State Observation}

\author{Armand Jordana$^{1}$, Bilal Hammoud$^{1}$, Justin Carpentier$^{2}$ and Ludovic Righetti$^{1,3}$
\thanks{$^{1}$Tandon School of Engineering, New York University, Brooklyn, NY
        {\tt\small aj2988@nyu.edu, bah436@nyu.edu, lr114@nyu.edu}}%
\thanks{$^{2}$Inria, Département d’informatique de l’ENS, \'Ecole normale supérieure, CNRS, PSL Research University, Paris, France
        {\tt\small justin.carpentier@inria.fr}}
        \thanks{$^{3}$Max-Planck Institute for Intelligent Systems, Tübingen, Germany}%
        \thanks{This work was in part supported by the European Union Horizon 2020 research and innovation program (grant agreement 780684), the National Science Foundation grants 1825993, 1932187, 1925079, 2026479, and the French government under management of Agence Nationale de la Recherche as part of the ”Investissements d’avenir” program, reference ANR-19-P3IA-0001 (PRAIRIE 3IA Institute), Louis Vuitton ENS Chair on Artificial Intelligence.
        }
}

\maketitle

\thispagestyle{empty}
\pagestyle{empty}

\begin{abstract}
In this letter, we study dynamic game optimal control with imperfect state observations and introduce an iterative method to find a local Nash equilibrium. 
The algorithm consists of an iterative procedure combining a backward recursion similar to minimax differential dynamic programming and a forward recursion resembling a risk-sensitive Kalman smoother. A coupling equation renders the resulting control dependent on the estimation. In the end, the algorithm is equivalent to a Newton step but has linear complexity in the time horizon length. Furthermore, a merit function and a line search procedure are introduced to  guarantee convergence of the iterative scheme.
The resulting controller reasons about uncertainty by planning for the worst case disturbances. Lastly, the low computational cost of the proposed algorithm makes it a promising method to do output-feedback model predictive control on complex systems at high frequency.
Numerical simulations on realistic robotic problems illustrate the risk-sensitive behavior of the resulting controller. 
\end{abstract}

\section{Introduction}

Doing output-feedback Model Predictive Control (MPC) while being robust to the estimator uncertainty is a notoriously difficult problem ~\cite{mayne2014model}. In optimal control, when the state measurement is partial and corrupted by noise, a common practice is to treat the estimation problem independently from the control and rely on the most likely outcome. However, in some scenarios, one might want to design a controller robust to the worst case estimation error. Such controllers can be obtained through dynamic game control~\cite{james1994risk, basar1991h}. As Mayne~\cite{mayne2014model} advocates, one meaningful way to properly take into account the measurement uncertainty is to use a minimax formulation linking control and estimation. 

This problem has been extensively studied in~\cite{copp2014nonlinear, copp2017simultaneous}, which shows that a specific dynamic game formulation leads to MPC approaches with bounded state trajectories and provides an explicit characterization of these bounds. However, the minimax problem was solved with an interior point method without taking into account the specific structure of the problem and the sparsity induced by time. In this work, we derive an explicit iterative solution that fully exploits sparsity, resulting in an algorithm that linearly scales with the time horizon length and which can be easily warm-started for use in MPC schemes \cite{mastalli20crocoddyl}.

In the linear dynamics and quadratic cost case, Jacobson~\cite{jacobson1973optimal} showed that dynamic game control is equivalent to risk sensitive control and derived a closed form solution. Later, Whittle~\cite{whittle1981risk} extended the results to the linear quadratic case with imperfect state observations. In~\cite{hammoud2021irisc}, a first iterative version of Whittle's solution was introduced to tackle the nonlinear risk sensitive problem with imperfect observations. However, the stochastic nature of the problem hindered the development of theoretical guarantees. Although, dynamic game and risk sensitive control are equivalent in the linear quadratic case~\cite{jacobson1973optimal}, this is no longer the case in the nonlinear setting~\cite{campi1996nonlinear}. Nonetheless, dynamic game control is tightly connected to robust and risk sensitive control. In~\cite{james1994risk, campi1996nonlinear}, James and Campi showed that dynamic game control can be interpreted as the limit case of risk sensitive control when noise tends to zero. Recently, Ba\c{s}ar~\cite{bacsar2021robust} presented a detailed overview of the connections between both problems in continuous time. Additionally, Ba\c{s}ar and Bernhard~\cite{basar1991h, bernhard1994minimax} established the connections between dynamic game control and $\mbox{H}^{\infty}$-Optimal Control both in the perfect and imperfect state information case. %

For nonlinear systems, estimating a state trajectory corresponding to some given measurements is usually intractable analytically. A common approach is to model the noise as Gaussian and to maximize the Maximum A Posteriori (MAP)~\cite{cox1964estimation}. 
If the dynamics are affine, then the problem can be solved analytically with the so-called Rauch–Tung–Striebel (RTS) smoother \cite{rauch1965maximum}. The RTS smoother is made of a forward recursion which resembles the Kalman Filter (KF) and a backward recursion. In the nonlinear case, iterative schemes are usually used. A popular choice is the iterative Kalman smoother which is equivalent to a Gauss-Newton method on the MAP \cite{bell1994iterated}. The estimation part of our proposed solution resembles a risk sensitive version of this smoother.

In optimal control or dynamic game control with perfect state information,
various numerical optimization algorithms have been developed to iteratively find solutions in the nonlinear case.
In optimal control, the most analogous to our work are Differential Dynamic Programming (DDP)~\cite{murray1984differential} and the stagewise implementation of the Newton's method~\cite{dunn1989efficient}. The stagewise Newton method is an exact implementation of Newton's method that exploits the specific structure of the Hessian matrix in order to scale linearly with the time horizon. DDP is an iterative algorithm that takes an update step on the control input by applying dynamic programming on a quadratic approximation of the value function. In~\cite{murray1984differential}, Murray showed that DDP is very similar to a Newton step and inherits its convergence properties. For dynamic game control with perfect state information, the seminal work from~\cite{mukai2000game, morimoto2003minimax} introduced minimax DDP showing that DDP could be extended to zero-sum two players games. Recently, \cite{di2021newton} further extended the concepts of stagewise Newton method and DDP to nonzero-sum games with an arbitrary number of players in the full information case. 

Despite having been widely studied theoretically, to the best of our knowledge, dynamic game control with imperfect state observations has not been approached from a numerical optimization point of view. In this work, we consider the general problem of dynamic game control with imperfect state observation and present a numerically efficient provably convergent algorithm to solve it. The solution presented generalizes commonly used techniques such as the extended Kalman smoother and Differential Dynamic Programming. The proposed solution fully exploits the sparsity of the problem and scales linearly with time, making it a promising method for model predictive control.

\textbf{Contributions.} In this paper, we derive a stagewise Newton method for dynamic game control with imperfect state observations. The proposed approach scales linearly in the time horizon. Furthermore, a merit function and a line search procedure are introduced to guarantee convergence. To our knowledge, this is the first iterative procedure with convergence guarantee for the dynamic game control problem with imperfect state observation. The proposed solver couples estimation and control by merging an iterative optimal control algorithm similar to minimax DDP and an iterative risk sensitive Kalman smoother. We illustrate the behavior of the resulting controller on two realistic robotic problems.

\textbf{Notations.}
The derivative of a function $f$ with respect to a vector $v$ is denoted by $f^v$, similarly for second order derivatives with respect to vectors $u,v$ is denoted as $f^{uv}$. If $(v_i)_{i\in \mathbb{N}}$ is a sequence of vectors, then $v_{k:t}$ is a vector concatenating $v_k \dots v_t$. $\mathbf{1}_{x \in A}$ is the indicator function which equals $1$ if $x \in A$ and $0$ otherwise. $I_n$ denotes the identity matrix of size $n$ by $n$.

\section{Problem Statement}

Similar to~\cite{copp2014nonlinear, copp2017simultaneous}, this work studies a special class of nonlinear dynamic games with imperfect state observation~\cite{james1994risk}.
Given a history of measurements  $y_{1:t}$, a history of control inputs, $u_{0:t-1}$ and a prior on the initial state~$\hat x_0 $, we aim to find a control sequence $u_{t:T-1}$ that minimizes a given cost~$\ell$ while an opposing player aims to find the disturbances $(w_{0:T}, \gamma_{1:t})$ that maximize this cost $\ell$ minus a weighed norm of the disturbances. Such a problem is formally written as:
\begin{align}\label{eq:main_problem}
&  \min_{u_{t:T-1}} \max_{w_{0:T}} \max_{\gamma_{1:t}}  \sum_{j=0}^{T-1} \ell_j(x_j, u_j) + \ell_T(x_T) \\
    & - \frac{1}{2\mu} \left( \omega_{0}^T  P^{-1} \omega_{0}  +  \sum_{j=1}^{t} \gamma_j^T  R_j^{-1} \gamma_j  + \sum_{j=1}^{T}   w_{j} ^ T Q_j^{-1} w_{j}   \right)\nonumber
\end{align}
\vspace{-0.6cm}
\begin{subequations}\label{eq:dynamics}
\begin{align} 
\mbox{subject to} \quad  x_{0} &= \hat{x}_0 + w_{0},  & \\
x_{j+1} &= f_j(x_j, u_j) + w_{j+1},  &   0 \leq j < T, \\
y_{j} &= h_j(x_j)+ \gamma_{j},  &  1\leq j \leq t.
\end{align}
\end{subequations}

\noindent where $\mu > 0$. $x_j$ is the state, $\omega_j$ the process disturbance, $\gamma_j$ the measurement disturbance, $T$ the time horizon, $t$ the current time. The transition model $f_j$, the measurement model $h_j$ and the cost $\ell_j$ are assumed to be $\mathcal{C}^2$. The measurement uncertainty $R_j$, the process uncertainty $Q_j$ and the initial state uncertainty $P$ are positive-definite matrices.

Interestingly, this problem encompass various formulations of control and estimation. If $t=0$ and if $w_0$ is fixed to zero, we recover dynamic game control with perfect state information. Additionally, if $t=0$, in the limit where $\mu$ tends to zero, we find the generic optimal control formulation \cite{campi1996nonlinear}. And lastly, if $t=T$ and if we consider all the cost $\ell_j$ to be null,  then, \eqref{eq:main_problem} is equivalent to maximizing the MAP.

\section{Method}
\subsection{First order conditions for a Nash equilibrium}
The main challenge in the problem formulation \eqref{eq:main_problem} is the equality constraint maintaining the dynamic feasibility. A popular approach \cite{morimoto2003minimax} is to derive a DDP-like algorithm by sequentially taking quadratic approximations of the value function recursion. However, a stagewise Newton step can readily be derived. One of the key features of the dynamic game \eqref{eq:main_problem} is that by changing the decision variable of the opposing player, one can transform the problem into an unconstrained one~\cite{james1994risk}. Indeed, we can use the equality constraints of Eq. \eqref{eq:dynamics} to replace disturbance maximization into a maximization over the state sequence. Then the problem loses its equality constraints and can be formulated as the search of the saddle point of
\begin{align}\label{eq:unconstrained_main_problem}
    J(&x_{0:T}, u_{t:T-1}) = \sum_{j=0}^{T-1} \ell_j(x_j, u_j) + \ell_T(x_T) \\
     & - \frac{1}{2\mu}  (x_0 - \hat{x}_0)^T  P^{-1} (x_0 - \hat{x}_0)   \nonumber \\
    & - \frac{1}{2\mu} \sum_{j=1}^{t} (y_j - h_j(x_j) )^T  R_j^{-1} (y_j - h_j(x_j) )   \nonumber \\
    & - \frac{1}{2\mu}  \sum_{j=0}^{T-1}    (x_{j+1} - f_j(x_j, u_j)) ^ T Q_{j+1}^{-1} (x_{j+1} - f_j(x_j, u_j)).   \nonumber
\end{align}
However, without convexity and concavity assumptions, we cannot aim at finding global solutions of the minimax problem. Hence, we restrict our attention to local Nash equilibrium, namely a point $(x_{0:T}^{\star}, u_{t:T-1}^{\star})$ such that there exits $\delta > 0$ such that for any $(x_{0:T} , u_{t:T-1} )$ satisfying
$|| x_{0:T} - x^{\star}_{0:T}  || < \delta $ and $|| u_{t:T-1} - u^{\star}_{t:T-1}  || < \delta $, we have
 \begin{align}
   J(x_{0:T}, u^{\star}_{t:T-1}) \leq  J(x^{\star}_{0:T}, u^{\star}_{t:T-1}) \leq J(x^{\star}_{0:T}, u_{t:T-1}). \label{eq:saddlepoint}
\end{align}
A standard approach to this problem is to search for a stationary point~\cite{bacsar1998dynamic}:
\begin{align}
    \begin{pmatrix}
    \frac{ \partial J}{ \partial x_{0:T}}(x^{\star}_{0:T}, u^{\star}_{t:T-1})\\[1mm]
    \frac{ \partial J}{ \partial u_{t:T-1}}(x^{\star}_{0:T}, u^{\star}_{t:T-1})
    \end{pmatrix} = 0
\end{align}
Interestingly, the change of decision variable in Eq. \ref{eq:unconstrained_main_problem} turned the problem into an unconstrained one but made no assumption on the structure of the cost. The only required assumption is that each disturbance is in a one to one map with the state at each time step.
\subsection{About the Linear Quadratic case}
In the linear quadratic case, Whittle has shown that under some conditions, the saddle point of \eqref{eq:saddlepoint} is global and can be computed analytically. More precisely, the order of the minimization and maximization in \eqref{eq:main_problem} can be interchanged, namely the lower value and upper value of the game are equal. Despite this result, one of the difficulties of the problem \eqref{eq:main_problem} is that it links estimation and control. One the major contribution of Whittle is the introduction of the notion of past stress and future stress showing that the KF and LQR principles can still be applied. In other words, the problem can be solved by performing a backward recursion on the controls and future states and a forward recursion on the past states. The future stress recursion can be interpreted as a value function recursion similar to LQR, while the past stress can be interpreted as a rollout of KF. Here, we use these two principles to efficiently solve iteratively the nonlinear case.

\subsection{A stagewise Newton's method}

In this section, a stagewise formulation of the Newton method to find a stationary point of $J$ is introduced. While a naive implementation of the Newton method would yield a complexity of $O (T^3) $, it is shown that the special structure of the Hessian induced by time can be exploited in order to obtain a linear complexity in time $O (T)$. In the perfect state observation case, \cite{dunn1989efficient} and \cite{di2021newton} derived a stagewise Newton method with a backward recursion on the controls. However, with imperfect state observation, it is no longer clear how to do this with only one recursion. Instead, we show that the principles introduced by Whittle can be applied.

To ensure that the proposed method is well defined and to guarantee convergence, we assume that the cost satisfies smoothness and non-degeneracy conditions required for the convergence of Newton's method \cite{nocedal1999numerical}. 
As the cost~\eqref{eq:unconstrained_main_problem} is unconstrained, the gradients and Hessian of the cost can readily be computed. At iteration $i$, given a guess $x_0^i,  x_1^i \dots, x_T^i, u_t^i \dots, u_{T-1}^i$, also referred as the nominal trajectory, the Newton step, denoted by $p$, satisfies
\begin{align} \label{Newton_step}
\hspace{-0.21cm} \begin{pmatrix}
     \dfrac{ \partial^2 J}{ \partial x_{0:T} \partial x_{0:T}} &  \dfrac{ \partial^2 J}{ \partial x_{0:T} \partial u_{t:T-1}}\\[3mm]
      \dfrac{ \partial^2 J}{ \partial u_{t:T-1} \partial x_{0:T}} &  \dfrac{ \partial^2 J}{ \partial u_{t:T-1} \partial u_{t:T-1}}
    \end{pmatrix} 
p = \begin{pmatrix}
     \dfrac{ \partial J}{ \partial x_{0:T}}\\[3mm]
     \dfrac{ \partial J}{ \partial u_{t:T-1}}
    \end{pmatrix}
\end{align}
\noindent
Here, $p =  \begin{pmatrix}
     p_{x_{0:T}}^T &
     p_{u_{t:T-1}}^T
    \end{pmatrix}^T$ where ${p_{x_{0:T}} \in \mathbb{R}^{(T+1) n_x}}$ is a stack of vectors $p_{x_k}\in \mathbb{R}^{n_x}$ with~$n_x$ being the dimension of the state space. Similarly, ${p_{u_{t:T-1}} \in \mathbb{R}^{(T-t) n_u}}$ is a stack of vectors $p_{u_k}\in \mathbb{R}^{n_u}$ with~$n_u$ being the dimension of the control space. 
To simplify the notations, we define an augmented Hessian of the cost that contains the second order derivatives of the dynamics for all~$k< T$.
\begin{align}
\bar{\ell}^{xx}_k  &= \ell^{xx}_k   +  \mu^{-1}  {w_{k+1}^i}^T Q_{k+1}^{-1} f^{xx}_k + \mu^{-1} \mathbf 1_{1\leq k \leq t} {\gamma_{k}^i}^T R_{k}^{-1} h^{xx}_k ,\nonumber\\
  \bar{\ell}^{xu}_k &= \bar{\ell}_k^{{ux}^T}= \ell^{xu}_k   +  \mu^{-1}  {w_{k+1}^i}^T Q_{k+1}^{-1} f^{xu}_k , \nonumber\\ 
\bar{\ell}^{uu}_k  &= \ell^{uu}_k   +  \mu^{-1} {w^i_{k+1}}^T Q_{k+1}^{-1} f^{uu}_k,
\end{align}
\noindent where the derivatives are evaluated at the current guess and where ${w_{k+1}^{i} \coloneqq  x_{k+1}^{i} - f_k(x_{k}^{i}, u_{k}^{i})}$ and  ${ \gamma_{k}^{i} \coloneqq  y_{k} - h_k(x_{k}^{i}) }$. Here, the second order derivatives of the dynamics are tensors. The exact definition of the tensor indexing and the tensor product is provided in the supplementary material~\cite{appendix}.

The next three propositions are analogous to the principles introduced by Whittle: the past stress recursion, the future stress recursion and the coupling of the past and future stress recursions. The first proposition, analogous to the future stress recursion, expresses every future state and control update steps as a function of $p_{x_t}$.
\begin{prop}[Future stress]
{In Equation \eqref{Newton_step}, the last ${(T-t)(n_x+n_u)}$ rows are equivalent to: 
\begin{align}
\forall k \geq t, \quad p_{u_k} =  G_k p_{x_k} + g_k& \label{control_forward_pass} \\
   p_{x_{k+1}} = \left( I  - \mu Q_{k+1} V_{k+1} \right)^{-1}  (&f^x_k p_{x_k} + f^u_k p_{u_k} \nonumber\\
  & + \mu Q_{k+1} v_{k+1}  - w^i_{k+1}  ) \nonumber
\end{align}
where $V_k$ and  $ v_k$ are solutions of the backward recursion:
\begin{align}
\Gamma_{k+1} &= I - \mu  V_{k+1} Q_{k+1} \label{control_backward_pass} \\
    Q_{uu} &=  \bar \ell^{uu}_k +  {f^u_k}^T \Gamma_{k+1}^{-1} V_{k+1} f^u_k \nonumber \\
    Q_{ux} &= \bar \ell^{ux}_k+ {f^u_k}^T \Gamma_{k+1}^{-1} V_{k+1} f^x_k \nonumber\\
 Q_{u} &=   \ell^{u}_k + {f^u_k}^T \Gamma_{k+1}^{-1} \left( v_{k+1} -  V_{k+1} w^i_{k+1} \right) \nonumber\\
    G_k &= - Q_{uu}^{-1} Q_{ux} \nonumber\\
    g_k &= - Q_{uu}^{-1} Q_u\nonumber\\
    V_k &=\bar \ell^{xx}_k + {f^x_k}^T \Gamma_{k+1}^{-1} V_{k+1} f^x_k + Q_{ux}^T G_k \nonumber \\
    v_k &= \ell^{x}_k  + {f^x_k}^T \Gamma_{k+1}^{-1} \left(v_{k+1} -  V_{k+1} w^i_{k+1} \right) + Q_{ux}^T g_k  \nonumber
\end{align}
with the terminal condition
\begin{align}
    V_T &=  \ell^{xx}_T, & v_T =  \ell^{x}_T.
\end{align}
}
\end{prop}
In those equations, $\mu$ intervenes only in $\Gamma_k$ and the augmented terms of the cost. Interestingly, $\Gamma^{-1}_k$ shifts, at each time step, the value function terms  $V_k$ and  $ v_k$. 
Then, the second proposition, analogous to the past stress recursion, expresses every past state update steps as a function of $p_{x_t}$.
\begin{prop}[Past stress]
In Equation \eqref{Newton_step}, if $t\geq1$, the first ${(t-1) n_x}$ rows are equivalent to: $ \forall k=0, \dots, t-1, $
\begin{align}
p_{x_k} &= E_{k+1}^{-1} \left(   {f^x_k}^T Q_{k+1}^{-1} ( w_{k+1}^i + p_{x_{k+1}} ) +  P_k^{-1}  \hat{\mu}_k  + \mu l^{x}_k \right)\label{estimation_backward_pass}
\end{align}
where $P_k$ and  $ \hat{\mu}_k$ are solution of the forward recursion:
\begin{align}
  E_{k+1} &=  P_k^{-1} +  {f^x_k}^T Q_{k+1}^{-1} f^x_k -  \mu  \bar{l}^{xx}_k \nonumber\\
\bar{P}_{k+1}  &= Q_{k+1} + f^x_k (P_k^{-1}  - \mu \bar{\ell}^{xx}_k)^{-1}  {f^x_k}^T \nonumber\\
    K_{k+1} &= \bar{P}_{k+1}   h_{k+1}^{x^T}  (R_{k+1}  + h^x_{k+1}  \bar{P}_{k+1}   h_{k+1}^{x^T} )^{-1} \nonumber\\
   P_{k+1} &= (I - K_{k+1} h^x_{k+1} )   \bar{P}_{k+1} \nonumber\\
   \hat{\mu}_{k+1}  &=   (I - K_{k+1} h^x_{k+1} ) ( f^x_k \hat{\mu}_k - w_{k+1}^i)  + K_{k+1} \gamma_{k+1}^i \nonumber \\
   &+  \mu P_{k+1} Q_{k+1}^{-1} f^x_k E_{k+1}^{-1} (  \bar{\ell}^{xx}_k \hat{\mu}_k +  \ell^{x}_k ) \label{estimation_forward_pass}
\end{align}
with the initialization
\begin{align}
    P_0 &= P, & \hat{\mu}_0 = \hat x_0 - x_0^i.
\end{align}
\end{prop}
Interestingly, if all the cost terms $\ell_j$ are zero and if $t=T$, the method is equivalent to a Newton method on the MAP. Furthermore, if the second order derivatives of the measurement function are omitted, then, the algorithm is equivalent to the iterative Kalman Smoother. Finally, the third proposition shows how both past stress and future stress recursions can be coupled to find the update step $p_{x_t}$.
\begin{prop}[Coupling]
In Equation \eqref{Newton_step}, the remaining rows (from $t n_x +1$ to $(t+1) n_x$) are equivalent to
\begin{align}\label{eq:coupling}
    p_{x_{t}} = \left(  P_t^{-1} - \mu    V_t \right) ^{-1} \left( P_t^{-1} \hat{\mu}_t +  \mu  v_t\right).
\end{align}
\end{prop}
In the limit case when $\mu$ tends to zero, the estimation and control are decoupled and we recover the usual certainty equivalence principle: $p_{x_t} = \hat{\mu}_t$. The algorithm is then equivalent to an iterative estimator and an iterative controller running independently. More precisely, at each iteration the controller uses the current estimate of the smoother.
\begin{proof}
The proof follows from the analytical derivations of the gradient and the Hessian of \eqref{eq:unconstrained_main_problem}, a forward induction from $0$~to~$t$ and a backward induction from $T$~to~$t$. The complete proof is provided in the supplementary material~\cite{appendix}.
\end{proof}

\setlength{\textfloatsep}{0pt}
\begin{algorithm}[!ht]
\footnotesize
\DontPrintSemicolon
\KwInput{$x_0^i,  x_1^i \dots, x_T^i, u_t^i \dots, u_{T-1}^i$}
\tcp{Estimation forward pass}
 $P_0 \gets P$, $ \hat{\mu}_0 \gets   \hat{x}_0 - x_0^i  $\;
\For{$k = 0, ... t-1$}
{
$P_{k+1}, \hat{\mu}_{k+1} \gets $  Eq. \eqref{estimation_forward_pass} \;      
}
\tcp{Control backward pass}
$V_T \gets \ell^{xx}_T $, $v_T \gets \ell^{x}_T $\;
\For{$k = T-1, ... t$}{
  $  V_k, v_k \gets $ Eq. \eqref{control_backward_pass} \;
}
\tcp{Estimation and control coupling}
$  p_{x_t} \gets \left(  P_t^{-1} - \mu    V_t \right) ^{-1} \left( P_t^{-1} \hat{\mu}_t +  \mu  v_t\right)$\;
\tcp{Estimation backward pass}
\For{$k = t-1, ... 0$}{
$p_{x_{k}}\gets   $ Eq. \eqref{estimation_backward_pass} \;
}
\tcp{Control forward pass}
\For{$k=t, \dots T-1$}{
$p_{u_k}, p_{x_{k+1}} \gets  $ Eq. \eqref{control_forward_pass}\;
}
\KwOutput{$p_{x_0},  p_{x_1} \dots, p_{x_T}, p_{u_t} \dots, p_{u_{T-1}}$}
\caption{Stagewise Newton step}\label{alg:newton_step}
\end{algorithm}

In the end, the update step, $p$, can be computed with a forward recursion on the past indexes, a backward recursion on the future indexes, a coupling equation, a backward recursion on the past indexes and a forward recursion on the future indexes. Algorithm \ref{alg:newton_step} summarizes those steps. Clearly, the complexity is linear in time. Instead of inverting a matrix of size ${((T+1) n_x+(T-t)n_u)}$, Algorithm \ref{alg:newton_step} only operates with matrices of size $n_x$ or $n_u$ and the number of operation is proportional to $T$.

\subsection{Line-search and convergence}
In the linear quadratic case, Algorithm \ref{alg:newton_step} is equivalent to Whittle's derivations and only one iteration is required to find a solution. However, in the general nonlinear case, several iterations of Newton's step are required. A common approach to guarantee the convergence of the overall iterative procedure is to introduce a line search and a merit function~\cite{nocedal1999numerical}. Given a guess at iteration $i$ and a direction~$p$, the next guess is defined by
\begin{align} \label{updaterule}
 \begin{pmatrix}
   x^{i+1}_{0:T} \\
     u^{i+1}_{t:T-1}
    \end{pmatrix}=
   \begin{pmatrix}
   x^{i}_{0:T} \\
     u^{i}_{t:T-1}
    \end{pmatrix}  + \alpha_i
    \begin{pmatrix}
     p_{x_{0:T}}\\
     p_{u_{t:T-1}}
    \end{pmatrix},
\end{align}
where the step length $\alpha_i$ is chosen in order to decrease the merit function. As advocated by Nocedal et al. \cite{nocedal1999numerical}, a Newton step provides a descent direction for the merit function
\begin{equation}\label{eq:merit}
f_{\mathcal{M}}(x_{0:T}, u_{t:T-1}) =  \frac{1}{2} \sum_{j=0}^T  \left\Vert \frac{ \partial J}{ \partial x_j}  \right\Vert^2  +  \frac{1}{2} \sum_{j=t}^{T-1}  \left\Vert \frac{ \partial J}{ \partial u_j}  \right\Vert^2 .
\end{equation}
This  merit function can be derived analytically. The exact derivations of each gradient are provided in the supplementary material~\cite{appendix}. By construction, the expected decrease of the direction $p$ derived by Algorithm \ref{alg:newton_step} is $ p^T \nabla f_{\mathcal{M}} =  - \vert \vert \nabla J \vert\vert_2^2 $ and the following convergence guarantees hold.
\begin{prop}
Assuming that the norm of the inverse of the Hessian of $J$ is bounded and that the step length $\alpha_i$ satisfies the Wolfe conditions for the merit function \eqref{eq:merit}, the sequence $(x^{i}_{0:T}, u^{i}_{t:T-1})_i$ defined by the update rule \eqref{updaterule} with steps from Algorithm \ref{alg:newton_step} is globally convergent to a stationary point of \eqref{eq:unconstrained_main_problem}. Furthermore, when the iterate is sufficiently close to the solution, the sequence has quadratic convergence.
\end{prop}

\begin{proof}This proposition is a direct consequence from the fact that Algorithm \ref{alg:newton_step} yields a Newton step. A detailed proof of the convergence guarantees of Newton method is provided in \cite{nocedal1999numerical}. 
\end{proof}

One may ask under which conditions the Hessian of $J$ is non-degenerate. Intuitively, large values of $\mu$ can make the problem ill-defined. Indeed, if the opposing player can choose large disturbances, then the controller might not be able to compensate. 
In the linear quadratic case, Whittle \cite{whittle1981risk} studied this maximum value for $\mu$ that makes the problem ill-defined. Although, we do not study this limit value in the nonlinear case, we note that analogously to the linear quadratic case, the algorithm is well defined if $\Gamma_{k}$ and ${P_k^{-1} - \mu \bar \ell_k^{xx}}$ are positive-definite which is the case when $\mu$ is small enough.
\subsection{About the cooperative case}

So far, only the case, $\mu > 0 $ has been considered, but, the case $\mu<0$ is also well defined. 
Indeed, the search for a stationary point of~$J$ can be done for an arbitrary sign of~$\mu$.
However, such a stationary point would now be a way to find a local minimum of $J$ with respect to the variables ${(x_{0:T}, u_{t:T-1})}$.
Interestingly, this scenario  can be interpreted as a cooperative scenario between the controller and the opposing player. In fact, the disturbances can be seen a second controller minimizing the cost, $\ell$, and maximizing the likelihood of the disturbances. Clearly, this change of sign does not affect the derivation of the stagewise Newton method. However, it can be noted that in that case, one can directly use the cost $J$ as a merit function.

\section{Experiments}
A Python implementation of the proposed method is available online \cite{appendix}. It is based on the Crocoddyl software~\cite{mastalli20crocoddyl}, a state of the art (risk-neutral) DDP solver which provides analytical derivatives of robot dynamics.
In this section, two numerical examples illustrate the proposed method.

\subsection{Planar quadrotor}

We study a quadrotor moving in a plane aiming to reach the position $(p_x\ \  p_y ) =(2\ \  0)$ starting at the origin without initial velocity. The state is ${x =(p_x \ \  p_y \ \ \theta \ \ \dot p_x \ \ \dot p_y \ \ \dot \theta )^T}$ where $\theta$ is the orientation of the quadrotor. The system dynamics is 
\begin{align}
m \ddot p_x &= - ( u_1 + u_2) \sin(\theta) ,\nonumber\\
m \ddot p_y &=  ( u_1 + u_2) \cos(\theta) - m g, \nonumber\\
J \ddot \theta &= r ( u_1 - u_2),
\end{align}

where the control input $u = (u_1 \ \ u_2 )^T \in \mathbb{R}^2$ represents the force at each rotor and $g$ is gravitational acceleration. An exponential cost models the presence of an obstacle 
\begin{align}
    \ell_k(x_k, u_k) &=  0.3 \exp \left(-10 ({p_x}_k - 1)^2 - 0.5 ({p_y}_k + 0.1)^2\right)\nonumber\\
    &+ 0.005 \Vert u_k - \bar u \Vert_2^2 + 0.05 (x_k - x_{\star})^T L (x_k - x_{\star}) \nonumber \\
    \ell_T(x_T) &= (x_T - x_{\star})^T L (x_T - x_{\star}),
\end{align}

\noindent where  $\bar u = \frac{m g}{2} ( 1 \ \ 1 )^T $, $x_{\star} = ( 2 \ \ 0 \ \ 0 \ \ 0 \ \ 0 \ \ 0 )^T$, and $L = \operatorname{diag}( 100 \ \ 100 \ \ 100 \ \ 1 \ \ 1 \ \ 1 )$. Only the position, $p_x$, $p_y$, and orientation~$\theta$ are part of the measurements. The integration and discretization of the model is done with Runge Kutta 4 with a time step of~$0.05$ and the total horizon, $T$, is $60$.
Furthermore, ${P_0 = Q = 10^{-5}  I_6}$, ${R = 10^{-4} \operatorname{diag} (1\ \ 1 \ \ 0.01 )}$ and $\hat x_0$ is the origin. A backtracking line-search is used -- a step is accepted if
$  {f_{\mathcal{M}}^{i+1}\leq     f_{\mathcal{M}}^{i}+ \alpha_i c   p_i^T f_{\mathcal{M}}^i}$
where ${f_{\mathcal{M}}^{i} = f_{\mathcal{M}}(x^{i}_{0:T}, u^{i}_{t:T-1}) }$ with ${c = \frac{1}{4}}$. The iterative process is stopped when the decrease in the merit function is lower than~$10^{-12}$. 

Figure~\ref{fig:quadrotor_sensitivity} illustrates the solution obtained by the solver for different values of $\mu$. The neutral
controller, the limit when $\mu$ tends to zero, aiming at minimizing the cost without accounting for disturbances, is solved using DDP. Interestingly, when $\mu$ is positive (the non-cooperative case), the opposing player chooses disturbances that will push the quadrotor towards to the obstacle and when $\mu$ is negative (the cooperative case), the opposing player chooses disturbances that will push the quadrotor away form the obstacle.
For this experiment, we found that the value of $\mu$ for which $\Gamma_k$ and ${P_k^{-1} - \mu \Bar \ell^{xx}_k}$ are no longer positive-definite matrices was around~$20$.
\begin{figure}[t]
    \centering
    \includegraphics[width=.9\linewidth]{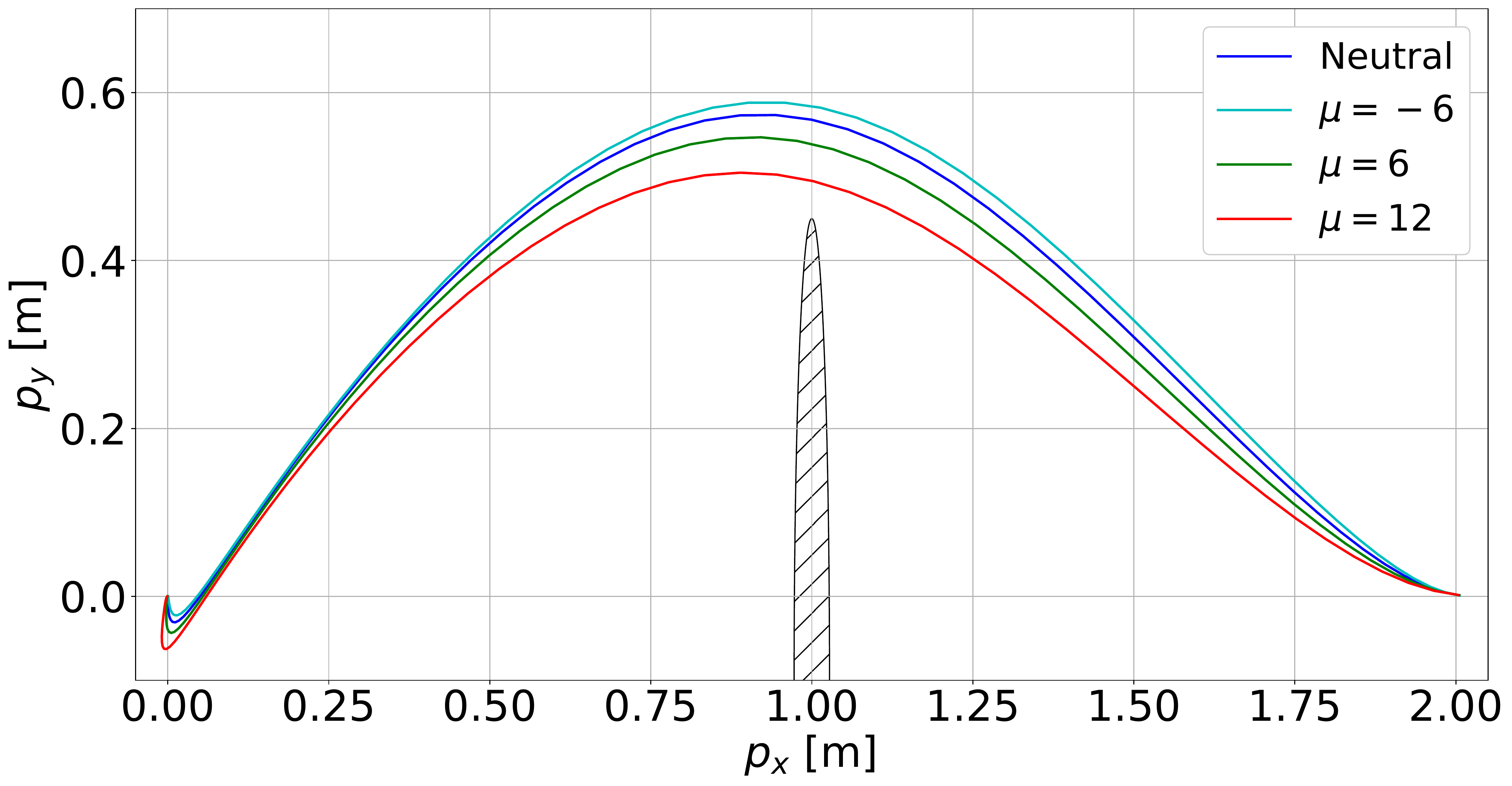}
    \caption{Initial plan for different values of $\mu$. The larger $\mu$ the more the controller plans to be pushed against the obstacle.}
    \label{fig:quadrotor_sensitivity}
\end{figure}

\begin{figure}[t]
    \centering
    \includegraphics[width=.9\linewidth]{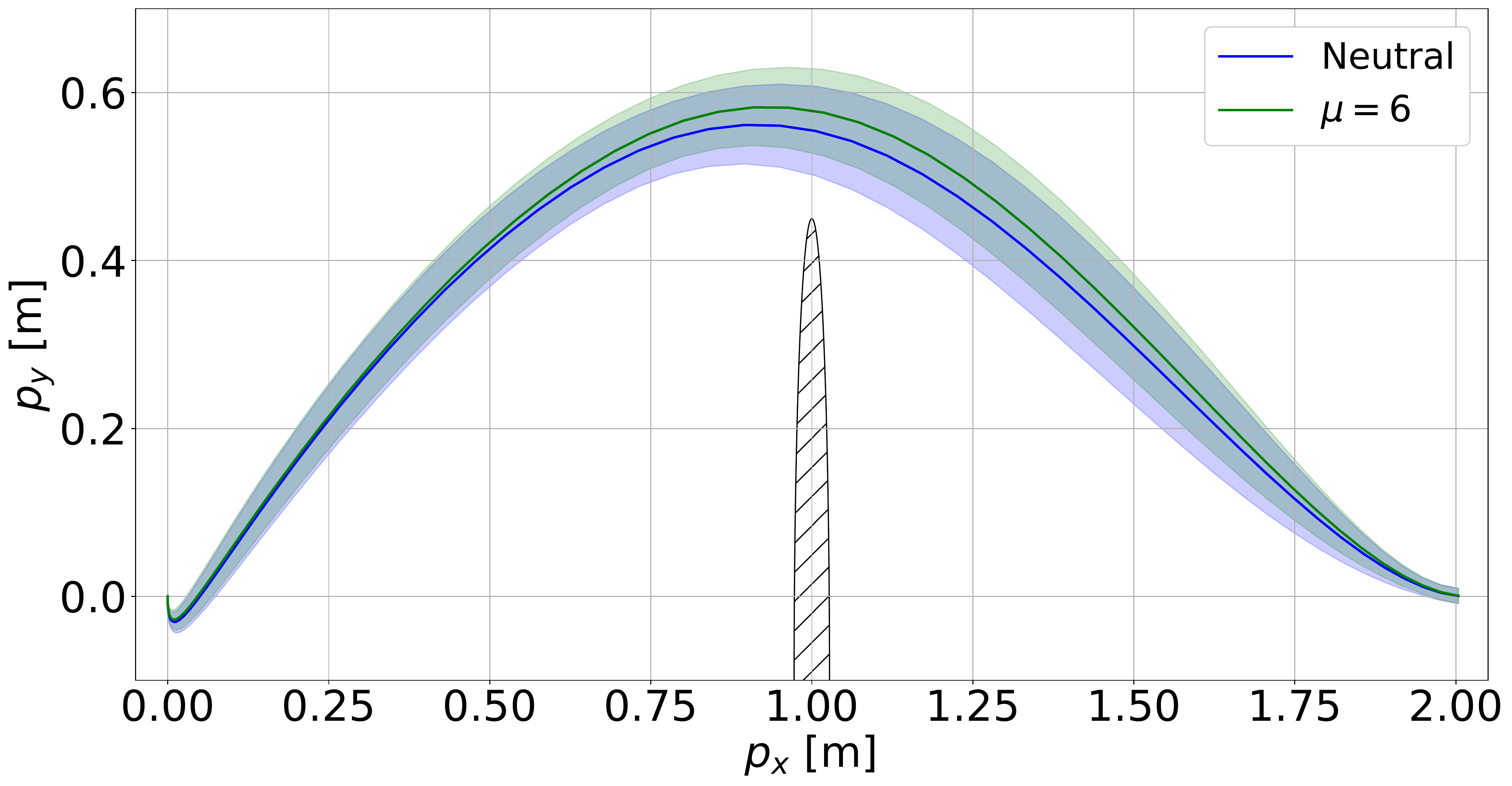}
    \caption{Average trajectory. Compared the neutral controller, the dynamic game controller ($\mu=6$) exhibits a risk sensitive behavior as it remains further from the high cost area representing the obstacle. }
    \label{fig:quad_mpc_traj}
\end{figure}

\begin{figure}[t]
    \centering
    \includegraphics[width=.8\linewidth]{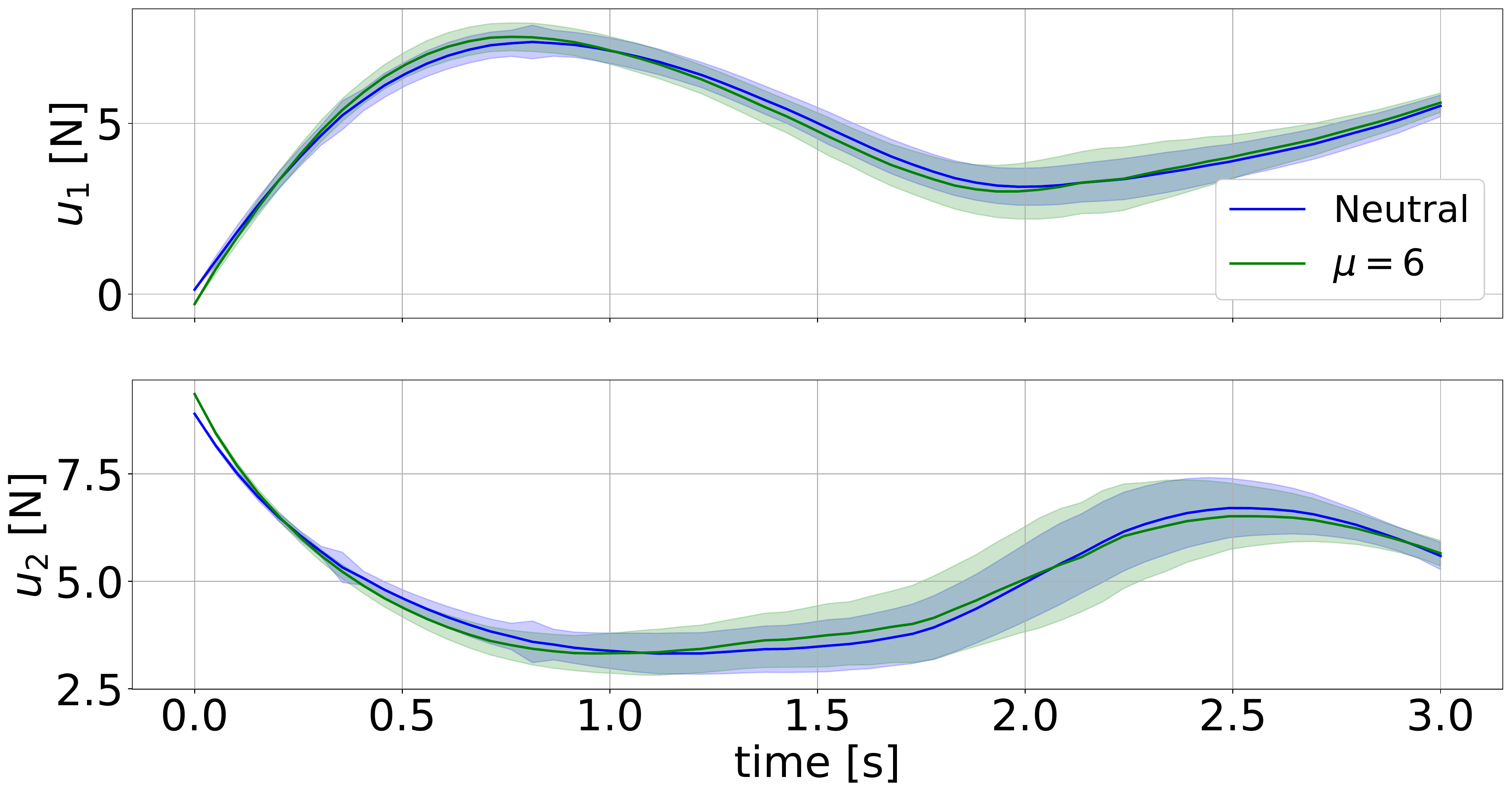} 
    \caption{Average control trajectories. Compared the neutral controller, the dynamic game controller ($\mu=6$) has a larger standard deviation.}
    \label{fig:quad_mpc_control}
\end{figure}

Next, we illustrate the risk-sensitive behavior of the resulting controller in closed loop in a simulation with disturbances following Gaussian distributions: ${x_0 \sim \mathcal{N}(\hat x_0, P_0)}$, ${w_k \sim \mathcal{N}(0, Q)}$ and ${\gamma_k \sim \mathcal{N}(0, R)}$.  We set $\mu=6$  and at each time step of the simulation, Algorithm~\ref{alg:newton_step} is run and $u_t$ is applied to the system.  Additionally, we compare to the neutral controller -- an iterative Kalman smoother is used for the filtering and the controller uses DDP with the last state estimate from the smoother as an initial condition.
Figure~\ref{fig:quad_mpc_traj} depicts the average and standard deviation over one thousand simulation. We can see that, in this MPC scheme, the dynamic game controller maintains a larger distance from the obstacle, resulting in safer behavior.
Figure~\ref{fig:quad_mpc_control} shows the distribution of control trajectories. Interestingly, the dynamic game controller has a larger standard deviation.

\subsection{An industrial robot}

In this example, we consider the 7-DoF torque-controlled KUKA LWR iiwa R820 14. The dynamics of the robot are provided by Pinocchio \cite{pinocchioweb}. The 14-dimensional state is composed of the joint positions and  velocities. We consider the following prior on the initial condition \mbox{$\hat x_0 = \begin{pmatrix} 0.1& 0.7& 0.& 0.7& -0.5& 1.5& 0.\end{pmatrix}^T$}.
The control input is a 7-dimensional vector of the torque applied on each joint. The goal is to move the end effector to a desired position, $ p_{\text{target}} = \begin{pmatrix} -0.4 & 0.3 & 0.7 \end{pmatrix}$ with the following cost:
\begin{align}
    \ell_k(x_k, u_k) &= 10^{-3}  \Vert x_k - \hat x_0 \Vert_2^2 + 10^{-6} \Vert u_k - \bar u(x_k) \Vert_2^2  \nonumber\\
    &+ 10^{-1} \Vert p_{\text{target}} - \bar p (x_k) \Vert_2^2  \nonumber \\
    \ell_T(x_T) &=\Vert p_{\text{target}} - \bar p (x_k) \Vert_2^2 + 10^{-3}  \Vert x_k - \hat x_0 \Vert_2^2 ,
\end{align}
where $\bar u(x_k)$ is the gravity compensation and $\bar p (x_k)$ the position of the end-effector.
For the measurement model, we assume that only the joints position are measured.
The initial control inputs $u_0, \dots u_{t-1}$ are generated with DDP. Given those $t$ initial control inputs, $t$ measurements are generated according to an undisturbed trajectory. More precisely, for $1\leq k \leq t$, the observations are defined by ${y_k = h_k(f^{(k)}(\hat x_0, u_{1:k-1}))}$ where $f^{(k)}(\hat x_0, u_{1:k-1})$ denotes the state value at the $k^{th}$ step integrated from the initial guess $\hat x_0$.
The horizon, $T$, is $100$ and $t$ is $5$ while the sensitivity parameter is set to $\mu= 1.2$. 
Here  $P_0 = Q = 0.01 \times I_{14}$ and $R = 0.5 \times I_{7}$. For this experiment, the second order derivatives of the dynamics were approximated to be null as the former are not provided by most standard rigid body dynamics libraries such as Pinocchio~\cite{pinocchioweb}. Note that this is a common practice for state of the art optimal control algorithms in robotics~\cite{mastalli20crocoddyl}. Our solver converges nevertheless, suggesting that this approximation might be used to scale to a large number degrees of freedom for real-time computations (e.g. MPC). 
In Figure~\ref{fig:kuka}, we plot the solution of the dynamics game solver compared to DDP in the end effector space, the shaded grey area represents the estimation part of the solution. We can see that the dynamic game controller plans that the disturbances will slow down the reaching task.

\begin{figure}[h]
    \centering
    \includegraphics[width=0.78\linewidth]{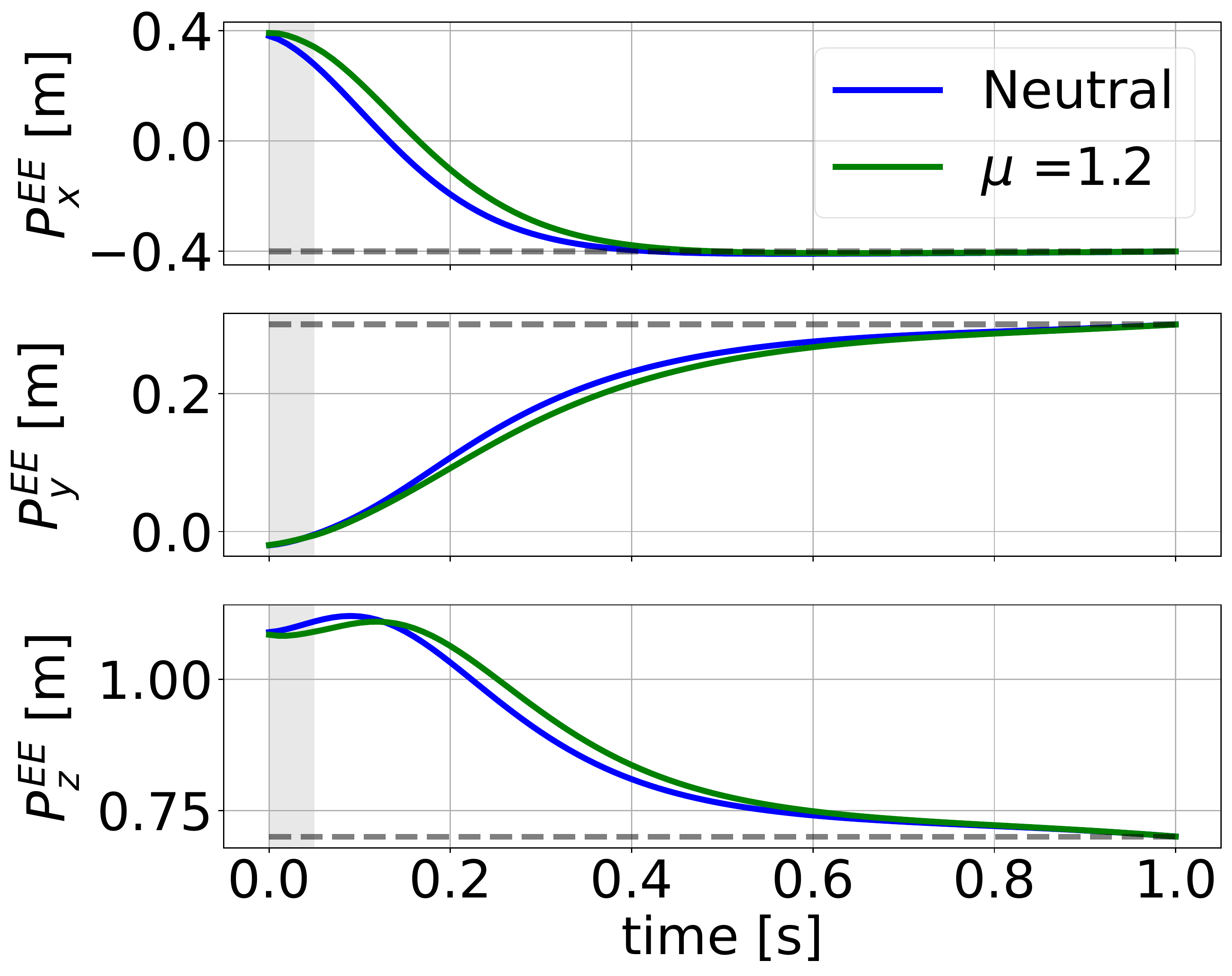}
    \caption{End-effector position $p^{EE}$ vs time. The dashed lines represent the target. }
    \label{fig:kuka}
\end{figure}

\section{Conclusion}
We introduced an iterative solver to find local Nash equilibrium of dynamic game with imperfect state measurements. The proposed algorithm is proven to be equivalent to Newton's method and benefits from its convergence properties while scaling linearly with the time horizon.


\bibliographystyle{IEEEtran}
\footnotesize
\bibliography{references}
\end{document}